\documentclass{amsart}

\usepackage{amssymb}
\usepackage{amsrefs}
\usepackage{overpic} 
\usepackage{graphicx} 
\usepackage{xcolor}
\usepackage{enumitem}  
\usepackage{mathrsfs}

\usepackage{verbatim}

\usepackage{enumitem}  
\usepackage[hidelinks]{hyperref}

\newcommand{\mb}{\mathbb}
\newcommand{\R}{\mb R}
\newcommand{\C}{\mb C}
\newcommand{\Pj}{\mb P}
\newcommand{\Z}{\mb Z}
\newcommand{\Q}{\mb Q}
\newcommand{\N}{\mb N}
\newcommand{\K}{\mb K}

\newcommand{\mc}{\mathcal}
\newcommand{\F}{\mc F}

\DeclareMathOperator{\sing}{sing}
\DeclareMathOperator{\I}{(I)}
\DeclareMathOperator{\II}{(II)}
\DeclareMathOperator{\III}{(III)}

\newtheorem{theorem}{Theorem} 
\newtheorem{proposition}{Proposition} 
 
\newtheorem{remark}{Remark} 
\newtheorem{lemma}{Lemma} 
 
\newtheorem{example}{Example}
\newtheorem{problem}{Problem}

\newtheorem{main}{Theorem}

\begin{document}

\title[Planar vector fields without invariant algebraic curves]{Planar vector fields without invariant algebraic curves}

\author[Gabriel Fazoli and Paulo Santana]{Gabriel Fazoli$^1$ and Paulo Santana$^2$}

\address{$^1$IRMAR-Université de Rennes, Rennes, France.}
\email{gabrielfazoli@gmail.com}

\address{$^2$IBILCE-UNESP, S. J. Rio Preto, SP, Brazil.}
\email{paulo.santana@unesp.br}



\begin{abstract}
    In this work we revisit and extend the method introduced by Lins Neto, Sad and Scárdua for detecting the non-existence of invariant algebraic curves other than some prescribed invariant nodal curve. We prove that, under the existence of a suitable example, the space of polynomial vector fields whose elements have the prescribed curve as their unique invariant algebraic curve is residual and of full measure. We apply this framework to Kolmogorov vector fields, showing that generically the coordinate axes are the unique invariant algebraic curves. Finally, we also refine existing characterizations related to Hilbert’s 16th problem, showing that if there exists a bound for the number of limit cycles of a vector field of degree n, then it can be attained by a vector field without algebraic limit cycles.
\end{abstract}

\maketitle

\section{Introduction}

Let $\mathbb{K}\in\{\mathbb{R},\mathbb{C}\}$ and consider the system of differential equations $\mathcal{X}$ on $\mathbb{K}^2$ 
\begin{equation}\label{01}
    \dot x=P(x,y), \quad \dot y=Q(x,y),
\end{equation}
with $P$, $Q\in\mathbb{K}[x,y]$, and the dot denoting the derivative in relation to the independent time variable $t\in\mathbb{K}$. For simplicity, we denote~\eqref{01} as $\mathcal{X}=(P,Q)$ and we may also refer to it as \emph{vector field}. The \emph{degree} of $\mathcal{X}$ is defined as $n=\max\{\deg P,\deg Q\}$.

Let $\gamma$ be one of the integral curves of $\mathcal{X}$. We say that $\gamma$ is \emph{algebraic} if it lies in an algebraic set, meaning that there exists a polynomial  $F\in\mathbb{C}[x,y]$, $\deg F>0$, such that the image of $\gamma$ is contained in $F^{-1}(\{0\})\cap\mathbb{K}^2$. When $F$ is taken to be irreducible, one can find a polynomial $K\in\mathbb{C}[x,y]$, called \emph{cofactor}, for which
\begin{equation}\label{02}
    P\frac{\partial F}{\partial x}+Q\frac{\partial F}{\partial y}=K\cdot F.
\end{equation}
Conversely, given a polynomial $F\in\mathbb{C}[x,y]$ (not necessarily irreducible), $\deg F>0$, we say that it is an \emph{invariant algebraic curve} of $\mathcal{X}$ whenever condition~\eqref{02} holds for some $K\in\mathbb{C}[x,y]$. The terminology \emph{invariant} reflects the fact that $F^{-1}(\{0\})\cap\mathbb{K}^2$ is invariant by $\mathcal{X}$. 

The study of invariant algebraic curves arises in several classical problems in the theory of dynamical systems, such as Darboux's problem (see~\cite{Darboux} and ~\cite{JouBook}*{Chapter 2}) and Poincaré's problem (see~\cites{Poincare, CLN, Carnicer}). These problems illustrate how the existence of algebraic curves imposes a certain rigidity on the geometry of dynamical systems. Therefore, it is natural to expect that a generic vector field admits no such curves. To make this intuition precise and specify the sense of ``generic'' involved, in this work we investigate the following problem.

\begin{problem}\label{Q1}
    Study the space of planar vector fields \emph{without} invariant algebraic curves.
\end{problem}

Besides this question seems to be first addressed for planar system of differential equations at~\cite{Landis}, the first remarkable result in this sense was provided by Jouanolou~\cite{JouBook}*{Chapter~$4$} in the context of \emph{holomorphic foliations}: it was proved that a \emph{generic} foliation on $\C\Pj^2$ does not have any invariant algebraic curve, where here \emph{generic} means that the set of foliations without invariant algebraic curves is an enumerable intersection of Zariski open subsets. Later Lins Neto~\cite{ALN1988}*{Theorem~$B$} proved that this space indeed contains an open and dense subset (with the natural euclidean topology). An alternative prove of this fact was provided at~\cite{Carnicer}.

Back to planar vector fields, several of the techniques developed by the aforementioned works still apply in this context. For instance, at~\cite{BerLun}*{Theorem~$5$} and~\cite{ALN1998}*{Theorem~$1.1$}, it was proved that in the space of complex planar polynomial vector fields of degree $n$, there is an open and dense set without invariant algebraic curves, and at~\cite{CouMen} a similar result was provided for planar systems with rational coefficients.

Regarding the real case, as far as we know, the only reference that resembles explicitly Problem~\ref{Q1} is~\cite{ALN1988}*{Theorem~$B'$}, which again deal with foliations on the real projective plane. For this reason, the first objective of this paper is to revisit the aforesaid works to properly address Problem~\ref{Q1} for the case of real planar vector fields. Then we apply our results on the celebrated Hilbert's 16th problem, providing an enhanced characterization of the vector fields realizing it (see Theorem~\ref{M2} below).

Rather than simply addressing Problem~\ref{Q1}, in this work we consider a more general question, showing that same techniques presented at \cites{ALN1988,ALN1998} can be generalized to a broader context. 

\begin{problem}\label{Q1'}
    Let $F$ be an algebraic curve on $\K^2$. Study the space of planar vector fields with no invariant algebraic curves other than $F$. 
\end{problem}

We now recall some notions used in the statement of our first main result. Given $n\in\mathbb{N}$, let $\mathfrak{X}_n(\mathbb{K}^2)$ be the set of polynomial systems~\eqref{01} of degree $n$. Given a topological space, a subset is \emph{residual} if it contains the intersection of countably many open and dense subsets. All topological spaces appearing in this work are \emph{Baire spaces} and thus every residual set will be dense~\cite{Munkres}*{Chapter~$8$}. Let $A\subset\mathfrak{X}_n(\mathbb{K}^2)$ be a measurable set. We say that it has \emph{full Lebesgue measure} if its complement has zero Lebesgue measure.

Given $F\in\mathbb{C}[x,y]$ a (not necessarily irreducible) polynomial, let $\mathfrak{X}_n^F(\mathbb{K}^2)$ be the family of vector fields of degree $n$ having $F$ as an invariant algebraic curve. We say that $F$ is \emph{nodal} if all its singularities are of normal crossing type, i.e. at each singularity of $F$ there are exactly two smooth branches of $F=0$ intersecting transversely. Moreover, we say that $F$ is nodal \emph{taking the line at infinity into account} if $F\cdot L_\infty$ is a nodal curve in $\mathbb{C}\mathbb{P}^2$, where $L_\infty=\mathbb{C}\mathbb{P}^2\setminus\mathbb{C}^2$. Let $\Upsilon_n^{d,F}(\mathbb{K})\subset\mathfrak{X}_n^F(\mathbb{K}^2)$ (resp. $\Upsilon_n^{\infty,F}(\mathbb{K})\subset\mathfrak{X}_n^F(\mathbb{K}^2)$) be the family of vector fields having no invariant algebraic curve of degree at most $d$ (resp. of any degree), other than $F$. In the case $F=1$, we simply omit $F$ from the notation. 

The strategy presented in \cite{ALN1988} to study foliations without invariant algebraic curves, and which we intend to apply here, is based on the \emph{Camacho-Sad Theorem} (see Equation~\ref{Eq: Camacho-Sad of a curve}). The precise definitions will be presented later in this text, but for the purpose of this introduction we only mention that we use Camacho-Sad Theorem to define \emph{Camacho-Sad obstructions} to the existence of an invariant curve other than a fixed one. These consist of a finite collection of numerical equalities for which a foliation must satisfy at least one in order to (possibly) possess an invariant curve other than a fixed curve $F$. 

\begin{main}\label{M3}
    Let $F\in\mathbb{C}[x,y]$ be nodal curve taking the line at infinity into account, and let $n\in \N$. Suppose there exists $\mathcal{X}_0\in\mathfrak{X}_n^F(\mathbb{\C}^2)$ with only Poincaré singularities, such that $\mathcal{X}_0$ does not satisfy any \emph{Camacho-Sad obstructions}. Then:
    \begin{enumerate}[label=$(\alph*)$]
        \item\label{a} $\Upsilon_n^{\infty,F}(\mathbb{C})$ contain an open and dense set.
    \end{enumerate}
    Moreover, suppose that there exists $\mathcal{X}_1\in\mathfrak{X}_n^F(\mathbb{\R}^2)$ with only simple singularities, such that $\mathcal{X}_1$ does not satisfy any \emph{Camacho-Sad obstructions}. Then:
    \begin{enumerate}[label=$(\alph*)$,resume]
        \item\label{b} $\Upsilon_n^{\infty,F}(\mathbb{R})$ is residual; and
        \item\label{c} for each $d\geqslant 1$, $\Upsilon_n^{d,F}(\mathbb{R})$ contain an open and dense set.
    \end{enumerate}
    Moreover, they all have full Lebesgue measure.
\end{main}

Throughout, we will present more precise statements for Theorem~\ref{M3} (see Theorems~\ref{T: lines invariant, complex case} and ~\ref{T6}). Notice that, by restricting the discussion above to the case $F=1$, we return to the discussion of Problem~\ref{Q1}. In this situation, the existence of the \emph{Jouanolou's foliations} implies the following characterizations.

\begin{main}\label{M1}
    For each $n\geqslant2$, the following statements hold.
    \begin{enumerate}[label=$(\alph*)$]
        \item $\Upsilon_n^\infty(\mathbb{C})$ contain an open and dense set;
        \item $\Upsilon_n^\infty(\mathbb{R})$ is residual; and
        \item for each $d\geqslant 1$, $\Upsilon_n^d(\mathbb{R})$ contain an open and dense set.
    \end{enumerate}
    Moreover, they all have full Lebesgue measure.
\end{main}

For completeness we included statement~$(a)$ in Theorem~\ref{M1}, although it was already established in~\cite{BerLun}*{Theorem~$5$} and~\cite{ALN1998}*{Theorem~$1.1$}. Statement~$(b)$ follows from~\cite{ALN1998}*{Theorem~$1.1$} in addition with some observations made in~\cite{ALN1988}*{Section~$3.4$}, and statement $(c)$ is essentially the real version of the argument presented by Jouanolou at~\cite{JouBook}*{Chapter~$4$}. However, to our knowledge, they have not been explicitly stated in the literature.

As another application of Theorem~\ref{M3}, we study \emph{Kolmogorov systems}, i.e. Systems~\eqref{01} for which both coordinate axes $x=0$ and $y=0$ are invariant. Such systems extend the classical Lotka–Volterra predator–prey equations, see~\cite{Karl} for details. We prove that, in general, the only invariant algebraic curves of a planar Kolmogorov system are the coordinate axes.

\begin{main}\label{M4}
    For each $n\geqslant2$, the following statements hold.
    \begin{enumerate}[label=$(\alph*)$]
        \item $\Upsilon_n^{\infty,xy}(\mathbb{C})$ contain an open and dense set;
        \item $\Upsilon^{\infty,xy}(\R)$ is residual; and
       \item for each $d\geqslant 1$, $\Upsilon^{d,xy}_n(\R)$ contain an open and dense set.
    \end{enumerate}
    Moreover, they all have full Lebesgue measure.
\end{main}

As anticipated, we also apply our results on Hilbert's 16th problem. To this end, given $\mathcal{X}\in\mathfrak{X}_n(\mathbb{R}^2)$ let $\pi(x)\in\mathbb{Z}_{\geqslant0}\cup\{\infty\}$ denote its number of \emph{limit cycles} (i.e. isolated periodic orbits). In his seminal lecture at the International Congress of Mathematicians in Paris, 1900, David Hilbert presented a list of problems for the $20$th century~\cite{Browder}. The second part of the $16$th problem concerns the study of limit cycles of planar polynomial vector fields. Hilbert asked whether there exists a uniform upper bound on the number of limit cycles for polynomial vector fields of degree $n$. More precisely, for $n\in\mathbb{N}$ let
\[
    \mathcal{H}(n):=\sup\{\pi(\mathcal{X})\colon \mathcal{X}\in\mathfrak{X}_n(\mathbb{R}^2)\}.
\]  
The second part of Hilbert’s $16$th problem can then be stated as the determination of an upper bound for $\mathcal{H}(n)$, a question that remains unsolved. In fact, even for quadratic systems it is still unknown whether $\mathcal{H}(2)<\infty$ or not. Nevertheless, considerable progress has been made in establishing lower bounds for $\mathcal{H}(n)$. For small values of $n$, as far as we know, the best results available are $\mathcal{H}(2)\geqslant 4$~\cites{ChenWang1979,Son1980}, $\mathcal{H}(3)\geqslant 13$~\cite{LiLiuYang2009}, and $\mathcal{H}(4)\geqslant 28$~\cite{ProTor2019}. More generally, it is known that $\mathcal{H}(n)$ grows at least on the order of $O(n^2\ln n)$~\cites{ChrLlo1995}. Regarding \emph{algebraic} limit cycles (i.e. limit cycles contained in invariant algebraic curves), it is know that under generic conditions, the maximum number is $1+(n-1)(n-2)/2$, see~\cites{HilAlg1,HilAlg3}.

In recent years the authors in~\cite{GasSan2025} provided a characterization of the vector fields realizing Hilbert's number $\mathcal{H}(n)$. Let $\Sigma_n\subset\mathfrak{X}_n(\mathbb{R}^2)$ be the set of \emph{structurally stable} vector fields, with hyperbolic limit cycles only. It is known that the Hilbert number is realizable by the elements of $\Sigma_n$.
\begin{theorem}[Theorem~$2$ of~\cite{GasSan2025}]\label{T4}
	For $n\in\mathbb{N}$, the following statements hold.
	\begin{enumerate}[label=(\alph*)]
		\item If $\mathcal{H}(n)<\infty$, then there is $\mathcal{X}\in\Sigma_n$ such that $\pi(\mathcal{X})=\mathcal{H}(n)$.
		\item If $\mathcal{H}(n)=\infty$, then for each $k\in\mathbb{N}$ there is $\mathcal{X}_k\in\Sigma_n$ such that $\pi(\mathcal{X}_k)\geqslant k$.
	\end{enumerate}
\end{theorem}

In this paper we shall not dip into the notion of structural stability of polynomial vector fields. For the interested reader, we refer to~\cite{GasSan2025}. We only remark that $\Sigma_n$ is open and dense in $\mathfrak{X}_n(\mathbb{R}^2)$.

Our last main result is an improving to Theorem~\ref{T4}. Consider $\mathcal{X}\in\mathfrak{X}_n(\mathbb{C}^2)$, and for a singularity $p\in\mathbb{C}^2$, let $\lambda_1,\lambda_2\in\mathbb{C}$ be its associated eigenvalues. We say that $p$ is \emph{simple} if $\lambda_1\lambda_2\neq0$ and $\lambda_1/\lambda_2\not\in\mathbb{Q}_{>0}$. If all singularities of $\mathcal{X}$ are simple, then we say that $\mathcal{X}$ is simple. Let $\mathcal{S}_n$ denote the set of simple vector fields. Let $\Lambda_n:=\Sigma_n\cap\Upsilon_n^\infty(\mathbb{R}^2)\cap\mathcal{S}_n$. We prove that the Hilbert number is realizable by the elements of $\Lambda_n$. 

\begin{main}\label{M2}
	For $n\in\mathbb{N}$, the following statements hold.
	\begin{enumerate}[label=(\alph*)]
		\item If $\mathcal{H}(n)<\infty$, then there is $\mathcal{X}\in\Lambda_n$ such that $\pi(\mathcal{X})=\mathcal{H}(n)$.
		\item If $\mathcal{H}(n)=\infty$, then for each $k\in\mathbb{N}$ there is $\mathcal{X}_k\in\Lambda_n$ such that $\pi(\mathcal{X}_k)\geqslant k$.
	\end{enumerate}
\end{main}

In particular, it follows from Theorem~\ref{M2} that if $\mathcal{X}$ (or $\{\mathcal{X}_k\}_{k\in\mathbb{N}}\}$) attains Hilbert's number, then there is no loss of generality to suppose that \emph{all its limit cycles are hyperbolic and non-algebraic}. Moreover, after considering its \emph{Poincaré compactification}, there is also no loss of generality to suppose that all singularities (including at the infinity) are hyperbolic and simple. 

The paper is organized as follows. Section~\ref{Sec2} contains preliminary material on foliations and systems of differential equations. The content of Lins-Neto work~\cite{ALN1988} is presented in Section~\ref{Sec3}, and its generalization is presented in Section~\ref{Sec4}. The main theorems are proved in Section~\ref{Sec5}, while Section~\ref{Sec6} is devoted to a further thought.

\section{Preliminary section}\label{Sec2}

In this section, we will associate systems of differential equations on $\mathbb{K}^2$ with foliations on the projective plane $\mathbb{K}\mathbb{P}^2$. To this end, we will make use of the concepts of vector fields and $k$-forms on manifolds, and the geometry of projective spaces. For details on these topics, we refer to~\cite{LeeBook}.  

For the purposes of this work, it suffices to view a foliation as phase portraits of analytic systems of differential equations, without the arrows given by the direction of its flow. For a more detailed treatment of the theory and for rigorous proofs of the facts presented here, we refer the interested reader to~\cites{CamALCBook,ALNBook,JouBook}. 

Finally, readers familiar with the notion of holomorphic foliations should be aware we adopted some non-standard conventions, as explained in Remarks ~\ref{Rmk: saturation} and ~\ref{Rmk: saturation 2}.

\subsection{From differential equations to foliations}\label{Sec2.1}

Consider a planar system of differential equations
\begin{equation}\label{1}
    \dot x=P(x,y), \quad \dot y=Q(x,y),
\end{equation}
with $P$, $Q\in\mathbb{K}[x,y]$. From now on we identify system~\eqref{1} with the vector field
\begin{equation}\label{2}
    \mathcal{X}:=P\frac{\partial}{\partial x}+Q\frac{\partial}{\partial y}.
\end{equation}
We say that $\mathcal{X}$ is \emph{saturated} if $\gcd(P,Q)\neq~1$.  Unless explicitly stated, we do not assume that $\mc{X}$ is saturated. The \emph{degree} of $\mc X$ is $\max\{\deg P,\deg Q\}$.

Let $\overline{\mathfrak{X}_n(\mathbb{K}^2)}$ denote the set of planar polynomial vector fields of degree at most $n$. Given $\mathcal{X}=(P,Q)\in\overline{\mathfrak{X}_n(\mathbb{K}^2)}$, let
\[
    P(x,y)=\sum_{i+j\leqslant n}a_{i,j}x^iy^j, \quad Q(x,y)=\sum_{i+j\leqslant n}b_{i,j}x^iy^j,
\]
and consider the isomorphism $\Phi\colon\overline{\mathfrak{X}_n(\mathbb{K}^2)}\to\mathbb{K}^{(n+1)(n+2)}$ given by
\begin{equation}\label{21}
    \Phi(\mathcal{X}):=(a_{0,0},a_{1,0},\dots,a_{1,n-1}a_{0,n},b_{0,0},b_{1,0},\dots,b_{1,n-1}b_{0,n}).
\end{equation}
The \emph{coefficients topology} of $\overline{\mathfrak{X}_n(\mathbb{K}^2)}$ is the topology given by the metric
\begin{equation}\label{15}
    \rho(\mathcal{X},\mathcal{Y}):=||\Phi(\mathcal{X})-\Phi(\mathcal{Y})||,
\end{equation}
where $||\cdot||$ is the standard norm of $\mathbb{K}^{(n+1)(n+2)}$. Remark that $\mathfrak{X}_n(\mathbb{K}^2)$ (i.e. the vector fields of degree $n$) is open and dense in $\overline{\mathfrak{X}_n(\mathbb{K}^2)}$. 

Let $\operatorname{sing}(\mc X)=\{r\in\mathbb{K}^2\colon P(r)=Q(r)=0\}$ be the \emph{singular set} of the vector field $\mc X$. If the foliation is saturated, then the singular set of $\mc{X}$ is finite, and when $\K=\C$ the converse also holds. Notice that $\mc X$ defines a decomposition of $\mathbb{K}^2\setminus\operatorname{sing}(\omega)$ into immersed sub-manifolds of codimension one locally given as solutions of system~\eqref{1}, and which we call the \emph{leaves} of $\mc X$. We refer to this decomposition as the \emph{foliation} $\F$ induced by $\mc X$ on $\K^2$. 

Let $S\subset \K^2$ be an immersed subvariety. We say that $S$ is \emph{invariant} by $\mc X$ if for every $r\in S$ we have $X(r) \in TS_r$. Remark then that not only the leaves are invariant by $\mc X$, but also the singular set of the foliation. In particular, we can have invariant codimension one subvarieties consisting only on singular points of $\mc X$ in case $\mc X$ is not saturated.

To the vector field~\eqref{2} we associate the 1-form 
\begin{equation}\label{3}
    \omega:=Pdy-Qdx.
\end{equation}
Notice that $\omega(\mc X)=0$ and thus, at each point $r\in \K^2$, the kernel of the linear map $\omega: T\K^2_r \rightarrow \K$ contains the vector $(P(r),Q(r))\in \K^2$ (and vice-versa). This allow us to study foliations on $\K^2$ both from the point of view of vector fields and from the point of view of 1-forms. In particular, all the definitions made above works also in the context of foliations defined by 1-forms.

Starting with a foliation $\F \in \overline{\mc X_n(\K^2)}$ defined by $\omega$ as above, it is simple to extend it to a (non necessarily saturated) foliation on $\mathbb{K}\mathbb{P}^2$ having the line at infinity $L_{\infty}$ as an invariant set, as follows. Let $[X:Y:Z]$ be the homogeneous coordinates of $\mathbb{K}\mathbb{P}^2$, and consider its charts $(U_x,\varphi_x)$, $(U_y,\varphi_y)$, $(U_z,\varphi_z)$ given by
\[
\begin{array}{lll}
    \displaystyle U_X=\{[X:Y:Z]\in\mathbb{K}\mathbb{P}^2\colon X=1\}, &\displaystyle \varphi_X([X:Y:Z])=\left(\frac{Y}{X},\frac{Z}{X}\right)=(x_1,y_1), \vspace{0.2cm} \\
    \displaystyle U_Y=\{[X:Y:Z]\in\mathbb{K}\mathbb{P}^2\colon Y=1\}, &\displaystyle \varphi_Y([X:Y:Z])=\left(\frac{X}{Y},\frac{Z}{Y}\right)=(x_2,y_2), \vspace{0.2cm} \\
    \displaystyle U_Z=\{[X:Y:Z]\in\mathbb{K}\mathbb{P}^2\colon Z=1\}, &\displaystyle \varphi_Z([X:Y:Z])=\left(\frac{X}{Z},\frac{Y}{Z}\right)=(x,y).
\end{array}
\]
To extend $\F$, we identify $\mathbb{K}^2=U_Z$ and calculate the expression of $\omega$ on the other charts. Let 
\[
    \varphi_{X,Z}:=\varphi_Z\circ\varphi_X^{-1}\colon U_X\to U_Z, \quad \varphi_{Y,Z}:=\varphi_Z\circ\varphi_Y^{-1}\colon U_Y\to U_Z,
\]
and notice that
\[
    \varphi_{X,Z}(x_1,y_1)=\left(\frac{1}{y_1},\frac{x_1}{y_1}\right), \quad \varphi_{Y,Z}(x_2,y_2)=\left(\frac{x_2}{y_2},\frac{1}{y_2}\right).
\]
Thus, the pullback of $\omega$ by $\varphi_{X,Y}$ is given by
\[
\begin{array}{ll}
   \displaystyle \varphi_{X,Y}^*(\omega) &\displaystyle= P\left(\frac{1}{y_1},\frac{x_1}{y_1}\right)d\left(\frac{x_1}{y_1}\right)-Q\left(\frac{1}{y_1},\frac{x_1}{y_1}\right)d\left(\frac{1}{y_1}\right) \vspace{0.2cm} \\
   &\displaystyle= \frac{1}{y_1^2}\left(Q\left(\frac{1}{y_1},\frac{x_1}{y_1}\right)-x_1P\left(\frac{1}{y_1},\frac{x_1}{y_1}\right)\right)dy_1+\frac{1}{y_1}P\left(\frac{1}{y_1},\frac{x_1}{y_1}\right)dx_1.
\end{array}
\]
Since $P$ and $Q$ are polynomials of degree at most $n$, we have that $\omega_X:=y_1^{n+2}\varphi_{X,Y}^*\omega$ is a polynomial $1$-form defined on $U_X$, and with the same kernel as $\varphi_{X,Y}^*\omega$. More precisely, let $P=P_0+\dots+P_n$, $Q=Q_0+\dots+Q_n$ be the decomposition of $P,Q$ in homogeneous polynomials and consider $L_1,M_1\in\mathbb{K}[x,y]$ given by
\begin{equation}\label{4}
\begin{array}{l}
    \displaystyle L_1(x_1,y_1)=y_1^nP\left(\frac{1}{y_1},\frac{x_1}{y_1}\right) = y_1^nP_0+y_1^{n-1}P_1(1,x_1)+\dots+P_n(1,x_1), \vspace{0.2cm} \\
    \displaystyle M_1(x_1,y_1)=y_1^nQ\left(\frac{1}{y_1},\frac{x_1}{y_1}\right) = y_1^nQ_0+y_1^{n-1}Q_1(1,x_1)+\dots+Q_n(1,x_1).
\end{array}
\end{equation}
From~\eqref{4} we conclude that the extension of $\F$ to $U_X$ is given by kernel of
\begin{equation}\label{5}
    \omega_X=\bigl(M_1(x_1,y_1)-x_1L_1(x_1,y_1)\bigr)dy_1+y_1L_1(x_1,y_1)dx_1.
\end{equation}
Similarly, the extension of $\mathcal{F}$ to $U_Y$ is given by the kernel of
\begin{equation}\label{6}
    \omega_Y=\bigl(x_2M_2(x_2,y_2)-L_2(x_2,y_2)\bigr)dy_2-y_2M_2(x_2,y_2)dx_2,
\end{equation}
with $L_2,M_2\in\mathbb{K}[x,y]$ defined analogously to~\eqref{4}. The triple $\{\omega,\omega_X,\omega_Y\}$ defines a foliation on $\K\Pj^2$, which we refer as the \emph{extension of $\F$ to $\K\Pj^2$}, and we also denote it by $\F$. Notice that the associated vector field of $\omega_X$ (resp. $\omega_Y)$ has the line $y_1=0$ (resp. $y_2=0$) as an invariant set, and thus $L_{\infty}$ is an invariant set of the extension of $\F$ to $\K \Pj^2$.

\begin{remark}\label{Rmk: saturation}
    The foliation $\F$ defined by $\{\omega,\omega_X,\omega_Y\}$ may be \emph{non-saturated}. In this case, our construction above differs from the standard foliation on $\K\Pj^2$ defined from the affine $\omega$, which is the \emph{saturated foliation} constructed from the triple $\{\omega,\omega_X,\omega_Y\}$. We do not saturate $\F$ by two reasons: first, $\omega$ itself may be non-saturated; and second, in this way $\F$ \emph{always} has the line at infinity invariant, even when it consists entirely of singular points. Another consequence of this non-standard definition is that the triple $\{\omega,\omega_X,\omega_Y\}$ do not depend on the degree of $\mc X$, rather it depends on $n$ such that $\mc X \in \overline{\mc X_n(\K^2)}$ and which in general can be greater than $\deg \mc X$. This convention makes the discussion that follows much simpler. 
\end{remark}

\subsection{Homogeneous coordinates of foliations}

A natural way of describing foliations on the projective plane is by means of \emph{homogeneous coordinates}. Let $\F$ be a (non necessarily saturated) foliation on $\K\Pj^2$. Given an affine chart $\K^2\subset \K\Pj^2$, we can restrict $\F$ to $\K^2$ such that the foliation is represented by a polynomial vector field as in~\eqref{2}. We say that $\mathcal{F}$ has \emph{degree} $n$ if the degree of the restriction of $\mathcal{F}$ to a generic affine chart~\eqref{2} is $n+1$.

Let $\pi: \K^3\setminus\{0\}\rightarrow \K\Pj^2$ be the natural projection $\pi(X,Y,Z) = [X:Y:Z]$. Given a foliation $\F$ of degree $n$ on $\K\Pj^2$, we consider then the \emph{pullback} of $\F$ to $\K^3\setminus\{0\}$. It is well known that this foliation is given by an 1-form
\begin{equation}\label{9}
    \Omega=\mathcal{P}dX+\mathcal{Q}dY+\mathcal{R}dZ,
\end{equation}
with $P,Q,R\in\mathbb{K}[X,Y,Z]$ homogeneous of degree $n+1$, and such that it satisfies the \emph{projective condition}
\begin{equation}\label{0}
    \Omega(\mathcal{E}) = 0,
\end{equation}
where $\mathcal{E}$ is the \emph{Euler's vector field}
\begin{equation}\label{10}
    \mathcal{E}=X\frac{\partial}{\partial X}+Y\frac{\partial}{\partial Y}+Z\frac{\partial}{\partial Z},
\end{equation}
defined over $\mathbb{K}^3$. Observe that~\eqref{0} can be written as $X\mathcal{P}+Y\mathcal{Q}+Z\mathcal{R}=0$. Hence, we have
\[
    \big(\mathcal{P},\mathcal{Q},\mathcal{R}\bigr)\perp(X,Y,Z),
\]
for each $(X,Y,Z)\in\mathbb{K}^3\setminus\{0\}$. In turn, this is equivalent to the existence of a third vector $(L,M,N)$, with $L,M,N\in\mathbb{K}[X,Y,Z]$, such that
\[
    \big(\mathcal{P},\mathcal{Q},\mathcal{R}\bigr)=(X,Y,Z)\times(L,M,N),
\]
where $\times$ denotes the usual cross product on $\mathbb{K}^3$. Hence, $\Omega$ is projective if, and only if, there are $L,M,N\in\mathbb{K}[X,Y,Z]$ such that
\begin{equation}\label{11}
    \mathcal{P}=YN-ZM, \quad \mathcal{Q}=ZL-XN, \quad \mathcal{R}=XM-YL.
\end{equation}
Notice that if $\mathcal{P},\mathcal{Q},\mathcal{R}$ are homogeneous of degree $n+1$, then $L,M,N$ are homogeneous of degree $n$. Moreover, $L,M,N$ are not uniquely determined, as they can be replaced by 
\begin{equation}\label{20}
    L_1=L+XH, \quad M_1=M+YH, \quad N_1=N+ZH, 
\end{equation}
with $H\in\mathbb{K}[X,Y,Z]$ homogeneous of degree $n-1$. 

\begin{example}\label{Example: extension of planar vector fields in homogeneous coordinates}
Let $\mc{X} \in \overline{\mc {X}_n(\K^2)}$ be a polynomial vector field~\eqref{2} on $\K^2$ and let us consider the extension of $\mathcal{F}$ to $\mathbb{K}^3\setminus\{0\}$ as described in the preceding section. Let us describe $\F$ in homogeneous coordinates.  We consider its original expression~\eqref{3} on $\mathbb{K}^2=U_Z$ and consider the map
\[
    \varphi_Z\circ\pi(X,Y,Z)=\left(\frac{X}{Z},\frac{Y}{Z}\right)=(x,y).
\]
defined in the open subset $\{Z\neq 0\} \subset \K^3$. Similarly to~\eqref{5} and~\eqref{6}, considering the pullback of $\omega$ and multiplying by $Z^{n+2}$, we obtain the 1-form
\begin{equation}\label{7}
   \Omega = -ZM\,dX+ZL\,dY+(XM-YL)dZ,
\end{equation}
with $L,M\in\mathbb{K}[X,Y,Z]$ given by,
\begin{equation}\label{8}
    L(X,Y,Z)=Z^nP\left(\frac{X}{Z},\frac{X}{Z}\right), \quad M(X,Y,Z)=Z^nQ\left(\frac{X}{Z},\frac{X}{Z}\right).
\end{equation}
Notice that each coefficient of~\eqref{7} is a homogeneous polynomial of degree $n+1$, and since $\Omega$ is not in general saturated, $n$ is not the degree of $\mc X$ (see Remark \ref{Rmk: saturation}).
\end{example}

Conversely, one can verify that any $1$-form $\Omega$ as in~\eqref{9} satisfying the projective condition~\eqref{0} is integrable, and thus defines a foliation on $\K^3$. Moreover, the projective condition also implies that the fibers of the projection $\pi: \K^3\setminus\{0\} \rightarrow \K\Pj^2$ are tangent to the leaves of the foliation defined by $\Omega$. Thus, $\Omega$ also induces a degree $n$ foliation on $\K\Pj^2$. Moreover, two homogeneous 1-forms $\Omega, \Omega'$ define the same foliation on $\K\Pj^2$ if and only if $\Omega = c \cdot \Omega'$ for some $c\in \K\setminus\{0\}$.

The representation of foliations on $\K\Pj^2$ by homogeneous 1-forms on $\K^3$ allows us to study the \emph{space of foliation on $\K\Pj^2$ of degree $n$}, which we denote by $\mathfrak{F}_n(\K)$. Indeed, let us denote by $\Lambda_{n+1}(\mathbb{K})$ the set of $1$-forms on $\mathbb{K}^3$ with coefficients given by homogeneous polynomials of degree $n+1$. Remark that the projective condition~\eqref{0} can be written as a system of linear homogeneous equations on the coefficients of $\mc{P},\mc{Q},\mc{R}$, and therefore the set of 1-forms satisfying~\eqref{0} defines a linear subspace $V_{n+1}(\mathbb{K})\subset\Lambda_{n+1}(\mathbb{K})$. Let $\pi\colon V_{n+1}(\mathbb{K})\setminus\{0\}\to\mathbb{P}(V_{n+1}(\mathbb{K}))$ be the natural projection. From the discussion above, each foliation on $\K\Pj^2$ corresponds to an unique class $[\Omega] \in \mathbb{P}(V_{n+1}(\mathbb{K}))$ and vice-versa. Therefore, we identify $\mathfrak{F}_n(\K)$ with $\Pj(V_{n+1}(\K))$ and endow it with the natural topology of projective spaces.  

\subsection{Foliations with the line at infinity invariant}

Let $\mathfrak{F}_n^\infty(\mathbb{K})\subset\mathfrak{F}_n(\mathbb{K})$ be the set of foliations having the line at infinity $L_\infty$ invariant. Observe that the process of extension of planar vector fields to foliation on $\K\Pj^2$ defined in Section ~\ref{Sec2.1} defines a map
\[
    \Psi\colon\Pj\big(\overline{\mathfrak{X}_n(\K^2)}\big) \rightarrow \mathfrak{F}_n^{\infty}(\mathbb{K}).
\]
Let $\F$ be a degree $n$ foliation on $\K \Pj^2$, represented in homogeneous coordinates by $\Omega$ with coefficients of degree $n+1$. Consider the map $\varphi\colon U_Z\simeq\mathbb{K}^2\to\mathbb{K}^3\setminus\{0\}$ defined by $\varphi(x,y)=(X,Y,1)$, and notice that the pullback of $\Omega$ by $\varphi$ is given by
\[
    \omega=\varphi^*\Omega=\bigl(L(x,y,1)-xN(x,y,1)\bigr)dy-\bigl(M(x,y,1)-yN(x,y,1)\bigr)dx.
\]
To $\omega$ we associate the vector field
\begin{equation}\label{13}
    \mathcal{X}=\bigl(L(x,y,1)-xN(x,y,1)\bigr)\frac{\partial}{\partial x}+\bigl(M(x,y,1)-yN(x,y,1)\bigr)\frac{\partial}{\partial y}.
\end{equation}
Replacing $P(x,y)=L(x,y,1)$, $Q(x,y,1)=M(x,y,1)$, and $R(x,y)=N(x,y,1)$ at~\eqref{13} we obtain
\begin{equation}\label{14}
    \mathcal{X}=(P-xR)\frac{\partial}{\partial x}+(Q-yR)\frac{\partial}{\partial y}.
\end{equation}
Observe that $P,Q,R$ are polynomials (not necessarily homogeneous) of degree at most $n$. Hence, the degree of $\mathcal{X}$ is at most $n+1$. Let $R=\overline{R}+R_n$ be the decomposition of $R$ into a polynomial $\overline{R}$ (not necessarily homogeneous) of degree at most $n-1$, and its homogeneous part $R_n$ of degree $n$.
\begin{proposition}[Lemma~$2$ of~\cite{ALN1988}]\label{Proposition: degree of the planar vector field induced by restriction}
    The following statements are equivalent.
    \begin{enumerate}
        \item $\mathcal{X}$ is of degree at most $n$;
        \item $L_{\infty}$ is invariant;
        \item $R_n=0$.
    \end{enumerate}
\end{proposition}

From Proposition~\ref{Proposition: degree of the planar vector field induced by restriction}, it follows that $\mathfrak{F}^{\infty}_n(\K)$ is a linear subspace of the space of foliations. In particular, we endow it with the inherited topology. 

\begin{proposition}\label{P1}
    $\Psi$ is a isomorphism.  
\end{proposition}

\begin{proof} 
    From Example~\ref{Example: extension of planar vector fields in homogeneous coordinates}, it follows that the extension of a planar vector field $\mc{X}$ to a homogeneous 1-form $\Omega$ depends analytically on the coefficients of $\mc{X}$, and therefore $\Psi$ is analytic. Conversely, considering the affine chart $\varphi: U_Z \simeq \K^2 \to \K^3\setminus\{0\}$, by Proposition \ref{Proposition: degree of the planar vector field induced by restriction} we have a well-defined map
    \begin{equation}
        \begin{split}
            \Psi^{-1}: \mathfrak{F}^{\infty}_n(\K) & \rightarrow \Pj\big(\overline{\mathfrak{X}_n(\K^2)}\big) \\
            \Omega & \mapsto \omega = \varphi^*\Omega,
        \end{split}
    \end{equation}
    which also depends analytically on the coefficients of $\Omega$, and thus it is analytic. An easy calculation shows that $\Psi$ and $\Psi^{-1}$ are inverse.  
\end{proof}

Notice from Proposition~\ref{P1} that if $\mathcal{U}\subset\mathfrak{F}_n^\infty(\mathbb{K})$ is dense (resp. open or residual), then $(\Psi\circ\pi)^{-1}(\mathcal{U})\subset\overline{\mathfrak{X}_n(\K^2)}$ is dense (resp. open or residual).

\subsection{Invariant algebraic sets}

In this section, we define the notion of invariant algebraic curves of foliations. We refer to an \emph{algebraic curve} either as a homogenenous polynomial $F\in\mathbb{K}[X,Y,Z]$ or as its associated \emph{algebraic set} $S=\{[r]\in\mathbb{K}\mathbb{P}^2\colon F(r)=0\}$. When $\K=\C$, there is essentially no distinction between these two points of view. However, when $\K=\R$, it is preferable to adopt the polynomial point of view, since an algebraic set of points may correspond to several polynomials.  

Let $\F$ be a foliation on $\K\Pj^2$ of degree $n$. We say that $F\in\mathbb{K}[X,Y,Z]$ is \emph{invariant} by $\mathcal{F}$ if
\begin{equation}\label{19}
    dF\wedge\Omega=F\alpha,
\end{equation}
for some $2$-form $\alpha$ of $\mathbb{K}^3$, with homogeneous coefficients of degree $n$. Furthermore, if $F=0$ is not contained in the singular set of the foliation (i.e. the set $\mathcal{P}=\mathcal{Q}=\mathcal{R}=0$), we say that $F$ is an \emph{algebraic solution} of $\F$.

\begin{remark}\label{Rmk: saturation 2}
    Since we are dealing with possibly non-saturated foliations, we allow invariant curves that consists entirely of singular points for the foliation. For example, given any foliation $\Omega$ and any invariant curve $F$, the (non-saturated) foliation $F\cdot \Omega$ always have the line $F=0$ as an invariant curve. 
\end{remark}

\begin{remark}
    Given a foliation $\F$, if an algebraic curve $F$ is invariant, then its associated algebraic set $S$ is also invariant by the foliation. Conversely, let $S\subset \K\Pj^2$ be an algebraic set invariant by the foliation, and let $F\in \K[X,Y,Z]$ be the homogeneous polynomial of minimal degree defining $S$. Then $F$ is invariant. This relates the \emph{algebraic} notion of invariance with the \emph{geometric} one.    
\end{remark}

Let $\mathcal{X}\in \overline{\mc{X}_n(\K^2)}$ be a vector field with degree at most $n$, and let $f\in \K[x,y]$ be a polynomial on the affine plane. Let $\F \in \mathfrak{F}_n^{\infty}(\K)$ be the corresponding degree $n$ foliation on $\K\Pj^2$ and let $F\in \K[X,Y,Z]$ be the homogenization of $f$. Using~\eqref{19} it is easy to see that $f=0$ is invariant by $\mathcal{X}$ if and only if $F$ is invariant by $\F$. In particular, $\F$ has no invariant algebraic curve (other than the line at infinity) if and only if $\mathcal{X}$ also has none. 
 
 \subsection{Singular set and separatrices}\label{Sec2.3}

The following discussion will take place mainly in the complex setting. Thus, whenever we have a real foliation $\F$, we will instead consider its \emph{complexification} $\F^{\C}$. 

Let $\F$ be a foliation on the projective plane, and as in~\eqref{9}, let $\Omega$ be a homogeneous 1-form representing it. Recall that the \emph{singular set} of $\F$ is given by
\[
    \mathrm{sing}(\F) = \{ [r]\in \C\Pj^2 ; \mc{P}(r)=\mc{Q}(r)=\mc{R}(r)=0\}.
\]
This definition is coherent with the definition for systems of differential equations, since if we consider $\mathcal{F}$ on an affine coordinate, the singularities of $\mathcal{F}$ in this coordinate coincide with the singularities of the respective vector field. It is not hard to verify that $\Omega$ is saturated if, and only if, the set of singularities is isolated. Moreover, it is a well-known fact that the set of saturated foliations is open and dense on $\mathfrak{F}_n(\K)$. 

The notion of \emph{separatrix} is used to describe different objects according to the framework under consideration, namely differential equations or foliations. For the former, we have the celebrated Markus-Neumann theorem, which states that the set of separatrices of a system, in addition to some regular orbits, completely characterizes the phase portrait of the system. For the definition of a separatrix in this framework and the proof of the Markus-Neumann Theorem, we refer to~\cite{Braun}.

In this paper, by \emph{separatrix} we will always refer to its definition according to foliation theory, which is as follows. First, separatrix in this context is a local concept around a singularity, and therefore we present the definition for a germ of foliation $\F$ on $(\C^2,0)$ with an isolated singularity at $0$. Let $B\subset (\C^2,0)$ be a germ of an analytic curve passing by $0$. We say that $B$ is a \emph{separatrix} of $\F$ if $B$ is invariant by the foliation. If $\F$ is defined by the germ of 1-form $\omega$ and $B=\{f=0\}$ for some germ of analytic function $f:(\C^2,0) \rightarrow (\C,0)$, to say that $B$ is a separatrix through $p$ is equivalent to say that there exists a germ of 2-form $\eta$ such that
\begin{equation}
    df \wedge \omega = f \cdot \eta.
\end{equation}
Observe that, in particular, given a foliation $\F$ on $\C\Pj^2$, an invariant curve $F$ is a separatrix throughout all the singularities of $\F$ over $F=0$. 

Let us recall that it is possible to associate a complex number $CS(\F,B)$ to a separatrix $B$ of $\mathcal{F}$ through an isolated singularity $p$, known as the \emph{Camacho-Sad index}. These indices were first introduced at~\cite{CamSad} in order to establish the celebrated Camacho-Sad Theorem, which asserts that through every singularity of a foliation passes a separatrix.  We shall not provide the definition of $CS(\F,B)$ here, since for the purposes of this work we only need to explicitly calculate it for \emph{simple singularities}, as it will be explained in the next section. For more details, see \cite{ALN1988}*{Section 1.3}.

Finally, in this work, we will extensively use the following fact. Let $\F$ be a foliation on $\C\Pj^2$ and let $S$ be an algebraic solution. Let $CS(\F,S)$ be the summation of the Camacho-Sad index of all the separatrices determined by $S$. Then,  
\begin{equation}\label{Eq: Camacho-Sad of a curve}
    CS(\F,S) = 3\deg(S) - \chi(S)+ \sum_{B\in S} \mu(B), 
\end{equation}
where $\chi(S)$ is the Euler characteristic of $S$ and $\mu(B)$ is the Milnor number of the separatrix $B$ on the respective singularity, see~\cite{ALN1988}*{Theorem A}.

\section{Foliations without algebraic solutions}\label{Sec3}

In this section, we present the main results obtained by Lins-Neto in~\cite{ALN1988}, which will be of great use for our goals in this paper. 

\subsection{More on singularities and existence of separatrices}\label{Sec3.1}

For the proceeding discussion, even if $\F$ is a real foliation, we will consider the set of singularities of its complexification $\F^{\C}$.  Let $\mathcal{F}\in\mathfrak{F}_n(\mathbb{\C})$, $p\in\operatorname{sing}(\mathcal{F})$, and $\lambda_1,\lambda_2\in\mathbb{C}$ be its associated eigenvalues in relation to some affine coordinate system. Following~\cite{ALN1988}*{Section~$3.1$}, we say that $p$ is \emph{non-degenerated} if $\lambda_1\neq0$ and $\lambda_2\neq0$. If $p$ is non-degenerated, we say that it is \emph{simple} if $\lambda:=\lambda_2/\lambda_1\not\in\mathbb{Q}_{>0}$. Finally, we say that it is of \emph{Poincaré type} if $\lambda\not\in\mathbb{R}_{>0}$. We remark that even though the eigenvalues $\lambda_1,\lambda_2$ depend on the choice of vector field defining $\F$ at the affine coordinate system, the quotient $\lambda$ does not.

We say $\mathcal{F}\in\mathfrak{F}_n(\mathbb{K})$ is non-degenerated (resp. simple or Poincaré) if all the singularities of $\mathcal{F}^\mathbb{C}$ are non-degenerated (resp. simple or Poincaré). Let
\begin{align*}
     \mathcal{N}_n(\mathbb{K}) &:=\{\mathcal{F}\in\mathfrak{F}_n(\mathbb{K})\colon \mathcal{F} \text{ is non-degenerated}\},  \\
     \mathcal{S}_n(\mathbb{K}) & :=\{\mathcal{F}\in\mathfrak{F}_n(\mathbb{K})\colon \mathcal{F} \text{ is simple}\}, \text{ and }  \\
     \mathcal{P}_n(\mathbb{K}) & :=\{\mathcal{F}\in\mathfrak{F}_n(\mathbb{K})\colon \mathcal{F} \text{ is Poincaré}\}. 
\end{align*}
Observe that $\mathcal{P}_n(\mathbb{K})\subset\mathcal{S}_n(\mathbb{K})\subset\mathcal{N}_n(\mathbb{K})$. If $\mathcal{F}\in\mathcal{N}_n(\mathbb{K})$, then $\mathcal{F}$ has exactly $N=n^2+n+1$ singularities $p_1,\dots,p_N$, see~\cite{ALN1988}*{Lemma~$4$}. For each one of its singularities $p_j$, we define the indices
\begin{equation}\label{18'}
    i_1(p_j,\mathcal{F})=\frac{\lambda_{j,2}}{\lambda_{j,1}}, \quad i_2(p_j,\mathcal{F})=\frac{\lambda_{j,1}}{\lambda_{j,2}}, \quad j\in\{1,\dots,N\},
\end{equation}
If $\mathcal{F}\in\mathcal{S}_n(\mathbb{K})$, then it follows from~\cite{CCD}*{Corollary 3.8} that $\mathcal{F}$ has exactly two separatrices $B_j^1,B_j^2$ through $p_j$, $j\in\{1,\dots,N\}$, tangent to the respective \emph{characteristic directions of $p_j$}, given by the eigenspaces associated with the eigenvalues $\lambda_{j,1},\lambda_{j,2}\in\mathbb{C}\setminus\{0\}$ of $p_j$, respectively. Moreover, in this case, one can calculate that the index defined at equation \eqref{18'} is indeed the \emph{Camacho-Sad index} of the respective separatrices, that is, 
\begin{equation}\label{18}
    CS(B_j^1,\mathcal{F})=\frac{\lambda_{j,2}}{\lambda_{j,1}}, \quad CS(B_j^2,\mathcal{F})=\frac{\lambda_{j,1}}{\lambda_{j,2}}, \quad j\in\{1,\dots,N\},
\end{equation}
see~\cite{ALN1988}*{Section~$3.1$} and~\cite{ALN1998}*{Section~$1$}.  

\subsection{Foliations without algebraic solutions} 

In order to prove the theorems of this section, we will need to consider the \emph{configurations of separatrices} of a given foliation, which we will make a precise definition in what follows. Let $\mathcal{A}$ denote the family of all the proper subsets of $\{(j,k)\in\mathbb{N}^2\colon 1\leqslant j\leqslant N,\, 1\leqslant k\leqslant 2\}$. Given a foliation $\F\in \mc N_n(\K)$, let us fix an enumeration of its singularities and their respective eigenvalues as before. Given a \emph{configuration} $A\in\mathcal{A}$, we can consider the associated \emph{configuration of indices}, which is the set $\{i_k(p_j,\F) \mid (j,k)\in A \}$. We also consider the summation of all these indices, which we denote by
\begin{equation}\label{17}
    \sigma_A(\mathcal{F})=\sum_{(j,k)\in A}i_k(p_j,\mathcal{F}).
\end{equation}
Recall that when $\F$ has only simple singularities (or, at least, the singularities that appear at the configuration $A$ are simple), we know that for each pair $(j,k)$ there exists a unique separatrix $B_{j,k}$ tangent to the characteristic direction of $\lambda_{j,k}$. Hence, in this case, the configuration $A\in \mc A$ could be seen as a \emph{configuration of separatrices}. Also in this case, $\sigma_A(\mathcal{F})$ turns out to be the sum of the Camacho-Sad indices of the respective separatrices.

Let $\F \in \mc S_n(\K)$. Given an irreducible algebraic solution $S$, let $A(S)\in \mc{A}$ be the configuration of the separatrices determined by $S$. In this case, from \eqref{Eq: Camacho-Sad of a curve} it follows that $\sigma_{A(S)}(\mc{F})$ must be a positive integer. As a consequence of this, we have the following.
\begin{theorem}[Theorem~$D$ of~\cite{ALN1988}]\label{T1}
    Let $n\geqslant2$ and $\mathcal{F}\in\mathcal{S}_n(\mathbb{C})$ be such that for any $A\in\mathcal{A}$, the number $\sigma_A(\mathcal{F})$ is not a positive integer. Then $\mathcal{F}$ has no algebraic solution.
\end{theorem}

In~\cite{ALN1988}, Lins-Neto uses Theorem~\ref{T1} to prove the following theorem.
\begin{theorem}[Theorem~$B$ of~\cite{ALN1988}]\label{T2}
    For all $n\geqslant2$, there exist an open and dense subset $\mathcal{U}\subset\mathfrak{F}_n(\mathbb{C})$ such that if $\mathcal{F}\in\mathcal{U}$, then $\mathcal{F}$ has no algebraic solution.
\end{theorem}

In a few words, to prove Theorem ~\ref{T2}, one must first notice that the set of Poincaré foliations $\mc P_n(\C)$ is open, dense and connected (see~\cite{ALN1988}*{Lemma~$5$}). Then, considering all possible configurations $A\in \mc A$, it is possible to use the locally well-defined and analytic functions $\sigma_A: \F \mapsto \sigma_A(\F)$ to determine an analytic subset $X\subset \mc P_n(\C)$ such that, if $\F \in \mc U := \mc P_n(\C)\setminus X$, then $\F$ has no algebraic solution. Finally, using the existence of the Jouanolou's foliations $\F_n$ (see~\cite{JouBook}*{p~$160$}), one conclude that $X$ has empty interior, that is, $\mc U$ is open and dense. This concludes the proof.

For the case of real foliations, we have the following result.

\begin{theorem}[Theorem~$B'$ of~\cite{ALN1988}]\label{T3}
    For all $n\geqslant2$, $\mc{S}_n(\R)\subset \mathfrak{F}_n(\R)$ is residual. Moreover, there exists a dense and relatively open subset $\mc{U}\subset \mc{S}_n(\R)$ such that if $\F \in \mc{U}$, then $\F$ has no algebraic solution.
\end{theorem}

The proof of Theorem~\ref{T3} is an adaptation of the proof of Theorem~\ref{T2} and we only remark the differences here. The first point is that, in this case, $\mc P_n(\R)$ is not open anymore, and thus we must consider the functions $\sigma_A: \F \rightarrow \sigma_A(\F)$ as locally well-defined and analytic on open neighborhoods of $\mc N_n(\R)$. Considering a small neighborhood $U\subset \mc N_n(\R)$ of $\mc{S}_n(\R)$, the same process allow us to define an analytic subset $X\subset U$ such that, if $\F \in \mc U := \mc S_n(\R)\setminus\mc S_n(\R)\cap X$, then $\F$ has no algebraic solution. We conclude as in Theorem~\ref{T2}. 

\section{Generalization}\label{Sec4}

\subsection{Foliation with a fixed invariant nodal curve}
Recall we say an algebraic curve $F\in \C[x,y,z]$ is \emph{nodal}  if all its singularities are of normal crossing type, i.e., at any singularity of $F$ there are exactly two branches of $F=0$ which intersect transversely. Writing $F=F_1\cdots F_k$ as a product of irreducible factors, this is equivalent to asking that each $F_i$ has only nodal singularities, and that any pair $F_i,F_j$ of distinct curves intersects transversely.

Let $\mathfrak{F}_n^F(\K)\subset \mathfrak{F}_n(\K)$ be the subset of foliations for which the curve $F$ is invariant. Observe that with no further assumptions we could have $\mathfrak{F}_n^F(\K)=\emptyset$ and in this case the discussion that follows would be meaningless. For this reason, we shall always assume that $\mathfrak{F}_n^F(\K)$ is non-trivial.

\begin{proposition}
    Let $F$ be a nodal curve, and let $n\geqslant 1$. Then $\mathfrak{F}_n^F(\K)$ is a linear subvariety of $\mathfrak{F}_n(\K)$.
\end{proposition}
\begin{proof}
    Let $\F_1,\F_2\in \mathfrak{F}_n^F(\K)$ be foliations given, respectively, by the homogeneous 1-form $\Omega_1,\Omega_2$. Since $F$ is invariant by both foliations, there exists 2-forms $\Theta_1,\Theta_2$ such that $dF \wedge \Omega_i = F \cdot \Theta_i$. Hence, for every constant $\alpha_1,\alpha_2\in \K$, we have that
    \[
    dF \wedge (\alpha_1\cdot \Omega_1 + \alpha_2\cdot \Omega_2) = F \cdot (\alpha_1 \cdot \Theta_1 + \alpha_2\cdot  \Theta_2),
    \]
    that is, the foliation given by $\alpha_1\cdot \Omega_1 + \alpha_2\cdot \Omega_2$ has also $F$ as an invariant curve. Therefore, $\mathfrak{F}_n^F(\K)$ is a linear subspace of $\mathfrak{F}_n(\K)$.
\end{proof}

Following the argument presented at~\cite{ALN1998}, we are able to generalize Theorems~\ref{T2} and~\ref{T3} for this situation. First, let
\[
\begin{split}
    & \mathcal{N}_n^F(\mathbb{K})=\mathcal{N}_n(\mathbb{K})\cap\mathfrak{F}_n^F(\mathbb{K}), \\ \vspace{0.2cm} 
    & \mathcal{S}_n^F(\mathbb{K})=\mathcal{S}_n(\mathbb{K})\cap\mathfrak{F}_n^F(\mathbb{K}),
    \text{ and }\\ \vspace{0.2cm} 
    & \mathcal{\mc P}_n^F(\mathbb{K})=\mathcal{P}_n(\mathbb{K})\cap\mathfrak{F}_n^F(\mathbb{K}),
\end{split}
\]     
be the set of non-degenerated, simple and Poincaré foliations on $\K\Pj^2$ with $F$ invariant. With no further assumptions, these sets may be empty. However, for the remainder of this section, we only consider the case when the appropriate set for the discussion is not empty. 

\begin{example} We are mostly interested in the case when $F = L_1\cdots L_k$ is the product of $k\leqslant n+2$ lines in general position. Considering the maximal case $k=n+2$, for each choice of numbers  $\lambda_1,\ldots, \lambda_{n+1}\in \K$, and $\lambda_{n+2} = - (\lambda_1 +\cdots +\lambda_{n+1})$, we consider the \emph{logarithmic foliation} generated by the homogeneous 1-form
    \[
     \Omega = L_1 \cdots L_{n+2} \cdot\left(\sum_{i=1}^{n+2} \lambda_i\cdot \frac{dL_i}{L_i} \right),
    \]
    Since the eigenvalues of the singularities depend analytically on $\lambda_1,\ldots,\lambda_{n+1}$ and they are not constant, a generic choice of them produces a foliation on the Poincaré domain.
\end{example}

Regarding the singularities of a foliation $\F \in \mc{N}_n^F(\C)$, we have:
\begin{enumerate}[label =(\Roman*)]
    \item\label{Item: cross} singularities that are singularities of the curve $F$. Let us call this \emph{singularities of type \ref{Item: cross}}, and let us denote by $N_{\I}$ the number of singularities of this type;
    \item\label{Item: line} at each $F_i$, there are other $N_{\II}^i$ singularities, where $N_{\II}^i$ is a number that depends only on $F$ and which we can calculate using the \emph{Gómez-Mont-Seade-Verjovsky index} (see~\cite{Brunella97}*{Section~$3$}). We call them \emph{singularities of type \ref{Item: line}};
    \item\label{Item: general} finally, there are $N_{\III}=(n^2+n+1)- N_{\I} - \sum N_{\II}^i$ singularities that does not belong to $F$, which we refer as the \emph{singularities of type \ref{Item: general}}. 
\end{enumerate}
Additionally, if $\F\in \mc{S}_n^F(\C)$ has only simple singularities, we have for singularities of type 
\begin{enumerate}[label =(\Roman*)]
    \item no other separatrix than the branches contained in $F$;
    \item at each singularity $p\in \sing(\F) \cap F_i$, there is one separatrix transverse to $F_i$; 
    \item at each singularity $p\in \sing(\F)\setminus\{F=0\}$, there are two separatrices not contained in $F$. 
\end{enumerate}

If $\mc A$ is the set of all configurations, we define $\mc A^F \subset 2^{\mc A}$ (the power set of $\mathcal{A})$ to be the subset of all configurations excluding the ones corresponding to the separatrices that are local branches of $F=0$. Hence, any configuration $A \in \mc A^F$ can be written as
\[
A = A_{\II}^1 \cup \cdots \cup A^k_{\II} \cup A_{\III}^1 \cup A_{\III}^2,
\]
where 
\begin{enumerate}[label = -]
    \item $A_{\II}^i$ is the configuration corresponding to separatrices of type \ref{Item: line} over points of the curve $F_i$; 
    \item $A_{\III}^2$ is the configuration corresponding to separatrices of singularities $p_j$ of type~\ref{Item: general}, such that 
    \[
    B_j^1 \in A \iff B_j^2 \in A;
    \]
    \item $A_{\III}^1$ is the configuration corresponding to separatrices of singulartities of type~\ref{Item: general}, such that 
    \[
    B_j^1 \in A \Rightarrow  B_j^2 \notin A \text{ and } B_j^2\in A \Rightarrow B_j^1\notin A. 
    \]
\end{enumerate}

As before, given a foliation $\F \in \mc N^F_n(\K)$ and a configuration $A\in \mc A^F$, we associate to it the complex number
\[
\sigma_A(\F)= \sum_{(j,k)\in A} i_k(p_j,\F),
\]
and if $\F \in \mc S^F_n(\K)$, then we can see it instead as the sum of the Camacho-Sad indices over the respective separatrices.

The following lemma is a consequence of~\cite{ALN1988}*{Lemma~$5$}.
\begin{lemma}\label{L: main lemma}
    Let $F$ be a nodal curve in $\K\Pj^2$. Then:
    \begin{enumerate}[label = -]
        \item If $\mc P_n^F(\C)\neq \emptyset$, then $\mc P_n^F(\C)$ and $\mc N^F_n(\C)$ are dense, open and connected subsets of $\mathfrak{F}^F_n(\K)$.        
        \item If $\mc N_n^F(\R)\neq \emptyset$, then $\mc N_n^F(\R)$ is an open and dense subset of $\mathfrak{F}^F_n(\R)$.
        \item If $\mc S^F_n(\R)\neq \emptyset$, then $\mc S^F_n(\R)$ is a residual subset of $\mathfrak{F}^F_n(\R)$.
    \end{enumerate}
    Moreover, given a foliation $\F \in \mc N^F_n(\K)$, there exists an open neighborhood $U\subset \mc N^F_n(\K)$ such that the singularities $p_i(\F)$ depend analytically on $\F$. If in addition $\F \in \mc S_n^F(\K)$ and shrinking $U$ if necessary, then for all configuration $A\in \mc A^F$, the map $\sigma_A: \F' \mapsto \sigma_A(\F')$ is a well-defined analytic function on $\F' \in U$. 
\end{lemma}

Let $\F \in \mc S_n^F(\K)$. Given $S$ an algebraic solution distinct from $F_1,\ldots,F_k$, observe that the configuration corresponding to the separatrices determined by $S$, which we still denote by $A(S)$, belongs to $\mc A^F$. Following as in the proof of~\cite{ALN1998}*{Proposition 2}, which in turn is just a consequence of equation~\eqref{Eq: Camacho-Sad of a curve}, we have the following.
\begin{proposition}\label{P3}
    Let $\F\in \mc S^F_n(\C)$ and let $S$ be an algebraic solution distinct from $F_1,\ldots,F_k$. Then:
    \begin{enumerate}
        \item\label{Item: A(S) is not the full set} $A(S)$ do not determine the full set of separatrices of $\mc A^F$; 
        \item\label{Item: intersects each Li at d points} $\# A(S)_{\II}^i=\deg(S) \cdot \deg(F_i)$ for all $1\leqslant i \leqslant r$; that is, $\big(\# A(S)_{\II}^i\big)/\deg(F_i)$ is a natural number independent of $i$, and 
        \item\label{Item: the equation} For each $1\leqslant i \leqslant r$,
        \[
            \sigma_{A(S)}(\F) = \left(\frac{\#A(S)_{\II}^i}{\deg(F_i)}\right)^2 -  \big(\#A(S)_{\III}^2\big).
        \]  
    \end{enumerate} 
\end{proposition}
A configuration satisfying the properties \eqref{Item: A(S) is not the full set} and \eqref{Item: intersects each Li at d points} above is called \emph{admissible}, and we denote the set of admissible configurations by $\mc A^F_0$. Given an admissible configuration $A\in \mc A^F_0$, we define $k(A) = \big(\# A(S)_{\II}^i\big)/\deg(F_i)$ and $\beta(A) = \# A_{\III}^2$. We use the proposition above in the form of the following lemma.

\begin{lemma}\label{L: no algebraic solution}
    Let $\mathcal{F}\in\mathcal{S}_n^F(\mathbb{C})$ be such that for any admissible configuration $A\in\mathcal{A}^F_0$, $\sigma_A(\mathcal{F}) \neq k(A)^2-\beta(A)$. Then $\mathcal{F}$ has no algebraic solution.
\end{lemma}

We remark that given a vector field $\mathcal{X}_0\in\mathfrak{X}_n(\mathbb{K}^2)$ and its associated foliation $\mathcal{F}_0=\Psi(\mathcal{X}_0)$, we have from Lemma~\ref{L: no algebraic solution} that if $\mathcal{X}_0$ has an invariant algebraic curve $S$, then its associated configuration $A(S)$ is admissible and satisfy $\sigma_A(S)(\mathcal{F}_0)=k^2(S)-\beta(S)$. Therefore, if every admissible configuration of $\mathcal{F}_0$ does not satisfy this last equation, then such $S$ can not exist. Hence, Lemma~\ref{L: no algebraic solution} can be used as a sufficient condition for the non-existence of invariant algebraic curves on vector fields.

\subsection{Foliations without algebraic solution other than a fixed nodal curve}

\begin{theorem}\label{T: lines invariant, complex case}
    Let $F\in \C[X,Y,Z]$ be a nodal curve, and let $n\geqslant 1$. If there exists $\F_0\in \mc P_n^F(\C)$ such that $\sigma_A(\F_0) \neq k(A)^2-\beta(A)$ for all admissible configuration $A\in \mc A^F_0$, then there exist an open and dense subset $\mathcal{U}^F\subset\mathfrak{F}_n^F(\mathbb{C})$ such that every $\mathcal{F}\in\mathcal{U}^F$ has no algebraic solution other than $F$.
\end{theorem}
\begin{proof}
    First, let us define an analytic subset $X\subset \mc P^F_n(\C)$ by defining it locally. Let $\F \in \mc P^F_n(\C)$ be a foliation, and let $U\subset \mc P^F_n(\C)$ be a small neighborhood of $\F$. Shrinking $U$ if necessary, for all admissible configuration $A \in \mc A_0^F$, the function
    \begin{equation*}
        \begin{split}
            \sigma_A: U & \rightarrow \C \\
            \F & \mapsto \sigma_A(\F)
        \end{split}
    \end{equation*}
    is well-defined and analytic by Lemma \ref{L: main lemma}. Thus $\sigma_A^{-1}(k(A)^2-\beta(A))$ is an analytic subset of $U$, for each $A \in \mc A^F_0$. Then, in a neighborhood of $\F$, we define the analytic subset
    \[
        X \cap U = \bigcup_{A \in \mc A^F_0} \sigma_A^{-1}\big(k(A)^2-\beta(A)\big),
    \]
    which clearly does not depend on the order we enumerate the singularities of $\F$. Therefore, this determines an analytic subset $X\subset \mc P^F_n(\C)$.
    
    Since $\F_0\notin X$, it follows that $X$ is closed with empty interior, and thus $\mc U = \mc P^F_n(\C)\setminus X$ is an open and dense. Finally, by Lemma \ref{L: no algebraic solution}, if $\F \in \mc U$, then $\F$ has no algebraic solution. This concludes the theorem.  
\end{proof}

A corresponding theorem also holds for real foliations.

\begin{theorem}\label{T6}
    Let $F\in \C[X,Y,Z]$ be a nodal curve and let $n\geqslant 1$. If there exists $\F_0\in \mc S_n^F(\R)$ such that $\sigma_A(\F_0) \neq k(A)^2-\beta(A)$ for all admissible configuration $A\in \mc A^F_0$, then there exist an relatively open and dense subset $\mathcal{U}^F\subset\mc S_n^F(\mathbb{R})$ such that every $\mathcal{F}\in\mathcal{U}^F$ has no algebraic solution other than $F$.
\end{theorem}
\begin{proof}
    By Lemma \ref{L: main lemma}, $\mc S^F_n(\R)$ is residual and $\mc N_n^F(\R)$ is open and dense. Proceeding as in the proof of Theorem~\ref {T: lines invariant, complex case}, we define an analytic variety $X'$ on an open neighborhood $V\subset \mc N_n^F(\R)$ of $\mc S^F_n(\R)$. Regarding $\mathfrak{F}_n^F(\R)$ as the subset of points of $\mathfrak{F}_n^F(\C)$ with real coordinates, remark that $X'$ is just the real part of the analytic variety $X$, and therefore, an interior point $\F \in X'$ must also be an interior point of $X$. 
    
    We define $\mc U = S_n^F(\R)\setminus X\cap \mc S_n^F(\R)$. Since $\mc F_0\notin X$, it follows that $X$ has empty interior, and therefore we conclude the proof of the theorem. 
\end{proof}

Recall that $\mathfrak{F}_n^\infty(\mathbb{K})$ denotes the set of holomorphic foliations of degree $n$ having the line at infinity invariant. Let $\mathcal{N}_n^\infty(\mathbb{R}),\mathcal{S}_n^\infty(\mathbb{R})$, and $\mathcal{P}_n^\infty(\mathbb{R})$ denote the respective subsets of nondegenerate, simple, and Poincaré foliations. As a corollary, we have the following.

\begin{theorem}[\cite{ALN1998}*{Theorem 4}]\label{T7} For all $n\geqslant 2$, there exists an open and dense subset $\mc U^{\infty}$ of $\mathfrak{F}^{\infty}_n(\C)$ (resp. a relatively open and dense subset of $\mc S_n^{\infty}(\R)$) such that for every $\F\in \mc U^{\infty}$, $\F$ has no algebraic solution other than the line at infinity. 
\end{theorem}
\begin{proof}
    The proof follows from the example provided at~\cite{ALN1998}*{Theorem 4}.
\end{proof}

We remark in Theorems~\ref{T6} and~\ref{T7} that as $\mathcal{S}_n^F(\mathbb{R})$ is residual in $\mathfrak{F}_n^F(\mathbb{R})$, the same holds for $\mathcal{U}^F$ and $\mathcal{U}^\infty$.

\subsection{Kolmogorov foliations}

We say that a foliation $\F$ on $\K \Pj^2$ is a \emph{Kolmogorov foliation} if it has the lines $x=0$, $y=0$, and $z=0$ as invariant sets. We denote the set of Kolmogorov foliations of degree $n$ by $\mathfrak{F}_n(\K)^{\rm{Kol}} \subset \mathfrak{F}_n(\K)$.

\begin{lemma}\label{Lemma: Kolmogorov}
    Let $F=xyz$ the product of three lines, $n\geqslant 2$, and $a_0\in \R_{<0}$ such that $\{1,a_0,1/a_0\}$ are $\Z$-linearly independent (for instance, $a_0 = -2^{1/4}$ or $a_0 = \exp(2\pi i/8)$). For all $b\in (\mathbb{D},0)$ (germ of real or complex disk), let $\F_b$ be the foliation defined by the homogeneous 1-form
    \begin{multline*}
        \Omega_b = yz(bx^{n-1}-y^{n-1}+z^{n-1})dx+xz(x^{n-1}-by^{n-1}-a_0z^{n-1})dy \\ + xy(-(b+1)x^{n-1}+(b+1)y^{n-1} + (a_0-1)z^{n-1})dz.
    \end{multline*}
    Then, there is $U\subset \mathbb{D}$, which is open and dense in the complex setting or have full measure on the real setting, such that for all $b\in U$, the foliation $\F_b$ has no admissible configuration $A \in \mc A^F_0$ satisfying $\sigma_A(\F_b) = k(A)^2-\beta(A)$.
\end{lemma}
\begin{proof}
    We start describing the singularities of $\Omega_b$. Let $\xi$ be a primitive $(n-1)$-th root of unity, and $a_0^{1/(n-1)}$ a solution of $t^{n-1}=a_0$. Recall that $\Omega_b$ has $n^2+n+1$ singularities, which we distinguish as follows.
    \begin{enumerate}[label = (\Roman*)]
        \item The `corner' points $[0:0:1]$, $[0:1:0]$, and $[1:0:0]$;
        \item The non-corner points in $x=0$, $y=0$, or $z=0$:
        \begin{enumerate}
            \item[$(x)$] $[0:\xi:\xi^i]$, with $1\leqslant i\leqslant n-1$;
            \item[$(y)$] $[\xi \cdot a_0^{1/(n-1)}:0:\xi^i]$, with $1\leqslant i\leqslant n-1$;
            \item[$(z)$] $[\xi:\xi^i:0]$, with $1\leqslant i\leqslant n-1$;
        \end{enumerate}
        \item The points in the affine chart $z=1$, but not in $xy=0$: 
        \[
            \big\{(x,y) \mid 1-y^{n-1}+bx^{n-1}=0 \text{ and } a_0-x^{n-1} + by^{n-1}=0\big\}.
        \]
        For $b=0$ the singularities are $(\xi^i\cdot a_0^{1/(n-1)},\xi^j)$, with $1\leqslant i,j\leqslant n-1$. In particular, for $b\approx0$ all singularities have multiplicity one.
    \end{enumerate}
    Considering their respective Jacobians, we have the following.
    \begin{enumerate}[label = (\Roman*)]
        \item At $[0:0:1]$ the quotients of the eigenvalues are $\{a_0, 1/a_0\}$, and in both $[0:1:0],[1:0:0]$ the quotients are $\{1+b, 1/(1+b)\}$. Then $[0:0:1]$ is simple, while $[0:1:0],[1:0:0]$ are simple if and only if $b\notin \Q$.  For the complex case, the singularities are Poincaré if $b\notin \R$. 
        \item The index on the direction $\ell=0$ and the index on the direction transverse to $\ell=0$, with $\ell\in\{x,y,z\}$, are given respectively by:
        \begin{enumerate}
            \item[$(x)$] $-(a_0+b)/(n-1)$ and $-(n-1)/(a_0+b)$;
            \item[$(y)$] $-(1+a_0b)/(n-1)a_0$ and $-(n-1)a_0/(1+a_0b)$;
            \item[$(z)$] $-(1-b)/[(n-1)(1+b)]$ and $-(n-1)(1+b)/(1-b)$;
        \end{enumerate} 
        Unless an enumerable number of choices for $b$, all these singularities are also simple. For the complex case, it is enough to take $b\notin \R$ small.
        \item For these singularities, if $b=0$ the quotients of eigenvalues are $\{a_0,1/a_0\}$. Therefore, for $b\approx0$, these singularities are simple in the real case and Poincaré in the complex case. 
    \end{enumerate}
    In conclusion, for a small real (resp. complex) disk $\mathbb{D}$, excluding an enumerable quantity of real numbers (resp excluding the real line), there is $U\subset \mathbb{D}$ of full measure (resp open and dense) such that $\F_b\in \mc{S}^F_n(\R)$ (resp $\F_b \in \mc{P}^F_n(\C)$) for all $b\in U$. Moreover, all the indices are well-defined analytic functions (even over non-simple singularities), and thus $\sigma_A: \F \mapsto \sigma_A(\F)$ is a well-defined analytic function on the whole disk. 

    The admissible configurations $A\in \mc{A}^F_0$ are composed by $\alpha$ simple branches, $\beta$ crossed branches and $k$ transverse branches over each line $x=0$, $y=0$, and $z=0$. We now prove that for each admissible configuration $A$, the set of foliations $\F_b$, $b\in (\mathbb{D},0)$ such that $\sigma_A(\F_b) = k(A)^2-\beta(A)$ is discrete. 

    Aiming a contradiction, suppose there exists an admissible configuration $A\in \mc{A}^F_0$ such that the set of foliations $\F_b$ such that $\sigma_A(\F_b) = k(A)^2-\beta(A)$ is not discrete. Since $\sigma_A$ is analytic, it implies $\sigma_A(\F_b) = k(A)^2-\beta(A)$ for all $b\in (\mathbb{D},0)$, and in particular $\sigma_A(\F_0) = k(A)^2-\beta(A)$. Considering the simple branches of $A$ (associated with singularities of type $\III$), suppose it has $\alpha_1$ characteristic directions corresponding to $a_0$ and other $\alpha_2$ characteristic directions corresponding to $1/a_0$. From the Jacobians at singularities of type $\II$ and $\III$ we conclude that $\sigma_A(\mathcal{F}_0)=k^2-\beta$ is equivalent to
    \[
        k^2-\beta =  \alpha_1\cdot a_0 + \alpha_2 \cdot \frac{1}{a_0} + \beta \cdot \left(a_0+\frac{1}{a_0}\right) + k\cdot \left(-(n-1)a_0 -\frac{n-1}{a_0} - (n-1)  \right).
    \]
    Since $\{1,a_0,1/a_0\}$ are linearly independent over $\mathbb{Z}$, this expression implies 
    \[
    \begin{split}
        k^2-\beta & = -k(n-1), \\
        0 & = \alpha_1 + \beta - k(n-1), \text{ and } \\
        0 & = \alpha_2 + \beta - k(n-1). 
    \end{split}
    \]
    From the first equation, we have $\beta-k(n-1) = k^2$. Replacing it on the second and third equations, we obtain $\alpha_1=\alpha_2 = - k^2$. The only non-negative integer solution is therefore $\alpha_1=\alpha_2=k=\beta=0$, and thus $A$ is not an admissible configuration, leading to a contradiction. 

    Therefore, for each admissible configuration $A\in \mc{A}^F_0$, there is only a discrete (and thus finite) set of $b\in (\mathbb{D},0)$ satisfying $\sigma_A(\F_b) = k(A)^2-\beta(A)$. This concludes the proof of the lemma.
\end{proof}

From the Lemma \ref{Lemma: Kolmogorov}, and Theorems \ref{T: lines invariant, complex case} and \ref{T6}, we conclude the following:

\begin{theorem}\label{T: Kolmogorv foliations}
    Let $\mathfrak{F}_n(\K)^{\rm{Kol}}$ be the space of Kolmogorov foliations at $\K\Pj^2$. Then, for each $n\geqslant2$, there exists an open and dense subset $U\subset \mathfrak{F}_n(\C)^{\rm{Kol}}$ (resp. a residual subset  $U\subset\mathfrak{F}_n(\R)^{\rm{Kol}}$) such that every $\F\in U$ has no algebraic solution other than $xyz=0$.
\end{theorem}

\section{Proof of the main theorems}\label{Sec5}

We recall that a real subset $A\subset\mathbb{R}^M$ is \emph{semi-algebraic} if it is the finite union of sets of form
\[
    \{x\in\mathbb{R}^M\colon f_1(x)=\dots=f_\ell(x)=0, \, g_1(x)>0, \dots, g_r(x)>0\},
\]
with $f_1,\dots,f_\ell,g_1,\dots,g_r\in\mathbb{R}[x]$, see~\cite{Algebric}*{Proposition~$2.1.8$}.

\begin{proof}[Proof of Theorem~\ref{M3}]
    Let $\overline{\mathfrak{X}^F_n(\mathbb{R}^2)}$ denote the affine space of vector fields of degree at most $n$ with $F$ invariant. Statements~\ref{a} and~\ref{b} follows from Proposition~\ref{P1}, Lemma~\ref{L: main lemma},  Theorems~\ref{T: lines invariant, complex case} and~\ref{T6}, together with the fact that $\overline{\mathfrak{X}^F_n(\mathbb{R}^2)}-\mathfrak{X}_n^F(\mathbb{K}^2)$ is a linear subspace of $\overline{\mathfrak{X}^F_n(\mathbb{R}^2)}$. Let us focus now on statement~\ref{c}. Let $\mathbb{R}^{N_1}$ and $\mathbb{R}^{N_2}$ be the spaces of polynomials $K\in\mathbb{R}[x,y]$, $F\in\mathbb{R}[x,y]$ of degree at most $n-1$, and $d$, respectively. Consider the space $Y_{n,d}=\overline{\mathfrak{X}^F_n(\mathbb{R}^2)}\times\R^{N_1}\times\R^{N_2}$ of triples $(\mathcal{X},K,F)$. Notice that $\mathcal{X}=(P,Q)$ has $G\neq F$ as an invariant algebraic curve with cofactor $K$ if, and only if,
    \begin{equation}\label{23}
        P\frac{\partial G}{\partial x}+Q\frac{\partial G}{\partial y}=KG, \quad G\neq F.
    \end{equation}
    Let $\Gamma_{n,d}\subset Y_{n,d}$ be the real semi-algebraic set defined by~\eqref{23}. Consider the projection $\Pi\colon Y_{n,d}\to\overline{\mathfrak{X}^F_n(\mathbb{R}^2)}$ given by $\Pi(\mathcal{X},K,F)=\mathcal{X}$, and let $A_{n,d}=\Pi(\Gamma_{n,d})$. Remark that $A_{n,d}$ is the set of vector fields with an invariant algebraic curve $G\neq F$ of degree at most $d$.  
    
    From \emph{Tarski–Seidenberg Theorem}~\cite{Algebric}*{Theorem~$2.2.1$} we have that $A_{n,d}$ is semi-algebraic in $\overline{\mathfrak{X}_n^F(\mathbb{R}^2)}\simeq\mathbb{R}^{N}$. Moreover, by ~\cite{Algebric}*{Theorem~$2.4.5$}, it can be decomposed as a finite union of semi-algebraic connected sets. Observe that none of these connected components has dimension $N$, otherwise $\overline{\mathfrak{X}_n^F(\R^2)}$ would have an open subset of vector fields containing an invariant curve other than $F$, which contracts ~\ref{b}. Therefore, the \emph{dimension} of $A_{n,d}$ is smaller than $N$.

    Observe now that the closure of $A_{n,d}$ is also semi-algebraic and of the same dimension~\cite{Algebric}*{Propositions~$2.2.2$, and~$2.8.2$} of $A_{n,d}$. Thus, $\overline{A_{n,d}}$ is closed and with empty interior. In particular, its complement is open, dense, and $(\overline{\mathfrak{X}_n^F(\R^2)}\setminus\overline{A_{n,d}}) \cap  \mathfrak{X}^F_n(\R^2) \subset  \Upsilon_n^{d,F}(\mathbb{R})$. This proves statement~\ref{c}. 
    
    Finally, the full measure of $\Upsilon_n^{\infty,F}(\mathbb{R}),\Upsilon_n^{d,F}(\mathbb{R})$ follows from the fact that each $\overline{A_{n,d}}$ has zero Lebesgue measure, while the full measure of $\Upsilon_n^{\infty,F}(\mathbb{C}),\Upsilon_n^{d,F}(\mathbb{C})$ follows by identifying $\overline{\mathfrak{X}_n^F(\mathbb{C}^2)}\approx\mathbb{C}^N\approx\mathbb{R}^{2N}$ and proceeding similarly.
\end{proof}

\begin{proof}[Proof of Theorem~\ref{M1}]
    It follows from Theorem~\ref{M3} and Theorem~\ref{T7}.
\end{proof}

\begin{proof}[Proof of Theorem~\ref{M4}]
    It follows from Theorem ~\ref{M3} and Theorem~\ref{T: Kolmogorv foliations}. 
\end{proof}

\begin{proof}[Proof of Theorem~\ref{M2}]
    It follows from Theorem~\ref{M1} (in addition with the property of simple singularities), Theorem~\ref{T4}, and the fact that $\Sigma_n$ is open and dense.
\end{proof}

\section{Further thoughts}\label{Sec6}

Since $\Upsilon_n^\infty(\mathbb{C})$ is open and dense, we have in the complex case that the family of vector fields with some invariant algebraic curve is not dense. The same property also holds in relation to $\Upsilon_n^d(\mathbb{R})$, for each $d\in\mathbb{N}$. On the other hand, the fact that $\Upsilon_n^\infty(\mathbb{R})$ is residual and of full measure does not prevent the set of real vector fields with at least one invariant algebraic curve to be dense. This rises a natural problem.

\begin{problem}
    Does $\Upsilon_n^\infty(\mathbb{R})$ contain an open and dense set? When does $\Upsilon_n^{\infty,F}(\R)$ contain an open and dense set? 
\end{problem}

\section*{Acknowledgments}

This work is supported by São Paulo Research Foundation, grants 2021/01799-9, and 2024/15612-6.

\end{document}